\documentclass[dvips,sts]{imsart}

\usepackage{amsthm,amsmath, amssymb}
\RequirePackage[colorlinks,citecolor=blue,urlcolor=blue]{hyperref}

\RequirePackage{hypernat}

\font\large=cmr10 scaled 1800
\font\tenmath=msbm10
\font\sevenmath=msbm7
\font\fivemath=msbm5
\newfam\mathfam \textfont\mathfam=\tenmath
\scriptfont\mathfam=\sevenmath \scriptscriptfont\mathfam=\fivemath

\usepackage{amsmath, moreverb, hyperref}
\hypersetup{colorlinks=true,citecolor=blue}

\usepackage{color}
\usepackage{amsfonts, fancybox}
\usepackage{amssymb}
\usepackage[american]{babel}
\usepackage{graphicx, pstricks, pst-node, color}
\usepackage{psfrag}\usepackage{subfigure}
\usepackage{flafter, bm, dsfont}
\usepackage[section]{placeins}

\startlocaldefs

\newcommand{\cB}{\mathcal{B}}

\newcommand{\cF}{\mathcal{F}}

\newcommand{\cS}{\mathcal{S}}

\newcommand{\fP}{\mathbf{P}}

\newcommand{\E}{\mathbf{E}}

\newcommand{\var}{{\bf Var}}
\newcommand{\fPu}{\fP_{\text{unif}}}
\newcommand{\fPp}{\fP_{\text{planted}}}

\makeatletter
\def\namedlabel#1#2{\begingroup
   \fPtect\def\@currentlabel{#2}%
   \label{#1}\endgroup
}
\makeatother

\newcommand{\BlackBox}{\rule{1.5ex}{1.5ex}}
\renewenvironment{proof}{\par\noindent{\bfseries\upshape
  Proof\ }}{\hfill\BlackBox\\[2mm]}
\newtheorem{theorem}{Theorem}[section]

\newtheorem{lemma}[theorem]{Lemma}

\begin{document}

\begin{frontmatter}

\title{Optimal Testing for Planted Satisfiability Problems}
\runtitle{Optimal Testing for Planted Satisfiability Problems}


 \author{Quentin Berthet\thanksref{t1,t2}\ead[label=e1]{qberthet@caltech.edu}}

\thankstext{t1}{The author thanks Philippe Rigollet, Emmanuel Abb\'e, Dan Vilenchick and Amin Coja-Oghlan for very helpful discussions.}
\thankstext{t2}{Partially supported by NSF grant CAREER-DMS-1053987 when the author was at Princeton University, and by AFOSR grant
FA9550-14-1-0098.}


 \address{{Department of Computing} \\
{ and Mathematical Sciences}\\
{California Institute of Technology}\\
{Pasadena, CA 91125, USA}\\
 \printead{e1}}

\runauthor{Q. Berthet}

\begin{abstract}
We study the problem of detecting planted solutions in a random satisfiability formula. Adopting the formalism of hypothesis testing in statistical analysis, we describe the minimax optimal rates of detection. Our analysis relies on the study of the number of satisfying assignments, for which we prove new results. We also address algorithmic issues, and give a computationally efficient test with optimal statistical performance. This result is compared to an average-case hypothesis on the hardness of refuting satisfiability of random formulas.
\end{abstract}



\tableofcontents

\end{frontmatter}

\section*{Introduction}

We study in this paper the problem of detecting a planted solution in a random $k$-{\sf SAT} formula of $m$ clauses on $n$ variables. This is formulated as a hypothesis testing problem: Given a formula $\phi$, our goal is to decide whether it is a typical instance, drawn uniformly among all formulas, or if it has been drawn such that it is guaranteed to be satisfiable, by planting a solution.

There is a resurgence in statistics of hypothesis testing problems, i.e., distinguishing null hypotheses with pure noise, against the presence of a structured signal in a high-dimensional setting. The seminal work of \cite{Ing82,Ing98,DonJin04}, on the problem of detecting sparse or weakly sparse signals in high dimension has inspired a wide literature of detection problems. Examples include  \cite{IngTsyVer10} in the context of sparse linear regression, \cite{AriCanDur11,ButIng13,AriVer13,MaWu13} for small cliques or communities in graphs and matrices,  \cite{AddBroDev10} for general combinatorial structured signals, and \cite{AriBubLug12,BerRig12,BerRig13} for sparse principal components of covariance matrices. These problems are combinatorial in nature, and the complexity of the class of possible signals (sparse vectors, cliques in a graph, small submatrices, or here the $n$-dimensional hypercube) has a direct influence on the statistical and algorithmic difficulties of the detection problem. 

Minimax theory gives a formal definition of the statistical complexity of a hypothesis testing problem, in terms of the sample size needed to identify with high probability the underlying distribution of given instances. 
It describes the interplay between the interesting parameters of a problem: sample size, ambient dimension, signal-to-noise ratio, sparsity, underlying dimension, etc.

This framework is particularly adapted to the study of random instances of $k$-{\sf SAT} formulas: a random formula $\phi$ can be interpreted as $m$ independent, identically distributed clauses, each on $k$ of the $n$ variables. The uniform distribution is equivalent to pure noise, the absence of signal. Planting a solution is equivalent to changing the distribution of the clauses, dependent on an assignment $x \in \{0,1\}^n$. This planted satisfying assignment is the signal whose presence we seek to detect. The optimal rate of detection will describe how large $m$ (the sample size) needs to be for detection to be possible, as a function of $n$ (the ambient dimension), and $k$, treated as a constant.

The properties of random instances of uniform $k$-{\sf SAT} formulas have been widely studied in the probability and statistical physics literature. Particular attention has been paid to the notions of satisfiability thresholds (sharp changes of behavior when the clause-to-variable density ratio $\Delta=m/n$ varies) \cite{AchPer04,AchMoo06,Coj09,CojPan12,Coj13,DinSlySun14}, maximum satisfiability \cite{AchPerNao03} geometry of the space of solutions \cite{AchPerNao03,AchRic06,AchCoj08,KrzMonRicc06,MonResTel09}, and concentration of specific statistics \cite{AbbMon10,AbbMon13}. The planted distribution has also been studied, often in order to create random instances that are known to be satisfiable, such as in \cite{BarHarLeo01,HaaJarKas05,AchGomKau00,AchJiaMoo04,AchCoj08,JiaMooStr05}, and at high density in \cite{AltMonZam06,CojKriVil07,FeiMosVil06}. Methods from statistical physics such as belief and survey propagation have been applied to this problem and rigorously studied \cite{BraMezZec02,MezParZec02,MezZec02,Coj10}. More recently, the algorithmic complexity (in a specific computational model) of estimating the planted assignment has been studied in \cite{FelPerVem13}.

Here, the use of tools from statistical analysis, such as the likelihood ratio and the total variation distance, highlights the importance of a specific statistic: the number of satisfying assignments. More specifically, we study its deviations from its expected value. Optimal rates of detection are obtained by proving new results concerning the concentration (or absence thereof) of this statistic. 
We address algorithmic issues by showing that the optimal rates of detection can be obtained by a newly introduced polynomial-time test. We also show the effect of choosing a different planting distribution on the detection problem, particularly on the optimal rates of detection. 

The following subsection introduces notations for $k$-{\sf SAT} formulas. Our hypothesis testing problem is formally described in Section~\ref{SEC:pdesc}. The optimal rates of detection are derived in Section~\ref{SEC:opttest}, and the problem of testing in polynomial time is addressed in Section~\ref{SEC:poly}. The effect on the detection rates of different choices for the planting distributions is studied in Section~\ref{SEC:alttest}.

\subsection*{Notations for  $k$-{\sf SAT} formulas}

Let $n$ and $m$ be positive integers. For all fixed positive integers $k$, we denote by $\cF_{n,m}^k$ the set of boolean formulas on $n$ variables that are the conjunction of $m$ disjunctions of $k$ distinct literals. Formally, for all $\phi \in \cF_{n,m}^k$, we have for all $x \in \{0,1\}^n$ 
$$\phi(x) = \bigwedge_{i=1}^m C_i(x)\, ,$$
where for all $i \in \{1,\ldots,m\}$, the clause $C_i$ is the disjunction of $k$ literals on $k$ distinct variables, i.e., the value of a variable or its negation
$$C_i(x) = \ell_{i,1}\vee \ldots \vee \ell_{i,k}\, , \; \ell_{i,j} \in \{x_1,\bar x_1,\ldots, x_n, \bar x_n \}\, , \text{and } \ell_{i,j} \notin \{\ell_{i,j'},\bar \ell_{i,j'}\} .$$

The $k$-{\sf SAT} problem (short for satisfiability) is the decision problem of determining whether a given formula $\phi$ is satisfiable, i.e., if there exists $x \in \{0,1\}^n$ such that $\phi(x)$ evaluates to  {\sf 'true'}. For a given $k$-{\sf SAT} formula $\phi$, we denote by $\cS(\phi)$ the set of satisfying assignments
$$\cS(\phi) = \big\{x \in \{0,1\}^n : \phi(x) = \, \text{{\sf 'true'}} \big\}\, ,$$
and by $Z(\phi) = |\cS(\phi)|$ the number of satisfying assignments for $\phi$. We often write $Z$ when it is not ambiguous. For a subset $S$ of $\{1,\ldots,m\}$, we define the sub-formula
$$\phi_S = \bigwedge _{i \in S} C_i \, .$$
The definition of satisfying assignments extends to single clauses  and sub-formulas in general, with the notations $\cS(C_i)$ and $\cS(\phi_S)$ for the set of assignments satisfying respectively, the clause $C_i$ or the formula $\phi_S$. We denote by ${\sf SAT}$ the set of satisfiable formulas: those with satisfying assignments.

\section{Problem description}
\label{SEC:pdesc}

We are interested in distinguishing two distributions on $\cF^k_{m,n}$, the \textit{uniform}, and \textit{planted} distributions. The uniform distribution, denoted by $\fP_{\text{unif}}$, is generated by independently selecting each clause uniformly from the $2^k {n \choose k}$ possible choices. The planted distribution, denoted by $\fP_{\text{planted}}$, is generated by randomly selecting an assignment $x^*$ uniformly among the $2^n$ elements of $\{0,1\}^n$, and then independently selecting all the clauses among the $(2^k-1) {n \choose k}$ clauses that are satisfied by $x^*$ (denoted by $\fP_{x^*}$). Each clause is given as $k$ literals, in a uniformly random order. 
We represent this as a hypothesis testing problem, on the observation $\phi \in \cF^k_{m,n}$
\begin{eqnarray*}
H_0&:& \phi \sim \fP_{\text{unif}} \\
H_1&:& \phi \sim \fP_{\text{planted}} = \frac{1}{2^n} \sum_{x \in \{0,1\}^n} \fP_x \, .
\end{eqnarray*}
It is also possible to consider the detection problem with composite alternative hypothesis over the $\fP_x$. Our formulation is equivalent to choosing a uniform prior over the planted assignments, and to consider the distribution $\fP_{\text{planted}}$, mixture of the $\fP_x$.
We will mention two regimes: the {\em linear regime}, when $m= \Delta n$, for some $\Delta >0$, usually the only one considered in the probability theory literature; and the {\em square-root regime}, when $m= C \sqrt{n}$, for some $C>0$, particularly relevant to the study of our statistical problem. We will often consider $m,n$ large enough, but will mainly focus on non-asymptotic results.\\

We define a test as a measurable function $\Psi:\cF^k_{m,n} \rightarrow \{0,1\}$, whose goal is to determine the underlying distribution of the observation $\phi$. We define the probability of error as the maximum of the probabilities of type I and type II error, formally
$$\fP_{\text{unif}}(\Psi(\phi)=1) \vee \fP_{\text{planted}}(\Psi(\phi)=0)\, .$$
This quantity is used here to measure the success of any test $\Psi$. We will consider that a test is successful when its probability of error is smaller than $\delta \in (0,1)$, considered fixed for the whole problem, such as $\delta = 0.05$.\\

We can make the simple observation that under the planted distribution, formulas are {\em guaranteed} to be satisfiable. This suggests to test satisfiability of the formula in order to solve the hypothesis testing problem. This test has a probability of error of type II equal to zero. Under the uniform distribution, the behavior of $\fP_{\text{unif}}(\phi \in \text{\sf SAT})$ has been extensively studied, and a phase transition has been shown to exist in the linear regime of $m=\Delta n$, from satisfiability to unsatisfiability, around some $\Delta_k$ close to $2^k \log(2)$. We refer to \cite{CojPan12,Coj13} and references therein for more information, as well as \cite{DinSlySun14} for a proof of the sharpness of the phase transition, for $k$ large enough.  
%
In this setting, when $\Delta >  \Delta_k$, the satisfiability test $\Psi_\text{\sf SAT} = \mathbf{1}\{\cdot \in \text{\sf SAT}\}$ has a probability of error going to 0, and when $\Delta <  \Delta_k$, the error will converge to 1 (entirely because of the probability of a type I error).

When thinking of the formula $\phi$ as a sequence of $m$ i.i.d. clauses, $m$ can be interpreted as the sample size, and the problem becomes easier when $\Delta$ increases. When $\Delta$ is too small, the probability of error of the test $\Psi_\text{\sf SAT}$ converges to 1. 
We see in the following section that this simple rate can be significantly improved. 

\section{Optimal testing}
\label{SEC:opttest}

In this section, we derive the optimal rate of detection for this problem, i.e., how large $m$ should be for a test to be able to distinguish with high probability the two hypotheses. We prove that the {\em likelihood-ratio test} is successful in the square-root regime, and show that it is information-theoretic optimal.

\subsection{Likelihood-ratio test}
A test based on the likelihood ratio between the two candidate distributions can distinguish between them with high probability, in the square-root regime. When $m \ge C\sqrt{n}$ for a specific constant $C$, the probability of error of the likelihood-ratio test is smaller than $\delta \in (0,1)$.

\begin{theorem}
\label{THM:uppersqr}
For all $k \ge 2$, positive $m,n$, denote $\Psi_\text{\sf LR}$ the likelihood-ratio test defined by
\begin{equation}
\label{EQN:lrtest}
\Psi_\text{\sf LR}(\phi) = \mathbf{1}\{Z(\phi) > \E_{\normalfont \text{unif}}[Z] \}\, .
\end{equation}
For any $\delta \in (0,1)$, there exists $\bar C_{k,\delta}>0$ such that for $m \ge \bar C_{k, \delta} \sqrt{n}$, for $m,n$ large enough, it holds
$$\fP_{\normalfont\text{unif}}(\Psi_\text{\sf LR}(\phi)=1) \vee \fP_{\normalfont \text{planted}}(\Psi_\text{\sf LR}(\phi)=0)\le \delta\, .$$ 
 
\end{theorem}

\begin{proof}
We first prove that the likelihood-ratio test has indeed form (\ref{EQN:lrtest}). For discrete distributions, the likelihood ratio is simply equal to the ratio of the two distributions. For all $\phi \in \cF^k_{m,n}$, it holds
$$
\frac{\fP_{\text{planted}}(\phi) }{\fP_{\text{unif}}(\phi) } = \frac{1}{2^n}\sum_{x \in \{0,1\}^n} \frac{\fP_{x}(\phi)}{\fP_{\text{unif}}(\phi)}\, .
$$
To compute the probabilities in the above ratios, we can interpret the drawing of $\phi$ by placing $m$ balls in $N=2^k {n \choose k}$ bins independently - if it has distribution $\fP_{\text{unif}}$ - or otherwise in the $N_k=(2^k-1){n \choose k}$ bins corresponding to clauses that are satisfied by $x$. Therefore, it holds for all $\phi$
\begin{equation*}
\frac{\fP_{x}(\phi)}{\fP_{\text{unif}}(\phi)} = \left\{
    \begin{array}{rl}
      0 &\; \text{if } x \notin \cS(\phi)\\
       \Big(\frac{N}{N_k}\Big)^m &\;\text{otherwise } 
    \end{array} \right.
\end{equation*}
It can  then be expressed in terms of $\mathbf{1}\{x \in \cS(\phi)\}$, and $N/N_k = 1/(1-2^{-k})$
\begin{eqnarray*}
\frac{\fP_{\text{planted}}}{\fP_{\text{unif}}}(\phi) &=& \frac{1}{2^n} \sum_{x \in \{0,1\}^n} \Big(\frac{N}{N_k}\Big)^m \mathbf{1}\{x \in \cS(\phi)\}\\
 &=& \frac{1}{\E_{\text{unif}}[Z(\phi)]} \sum_{x \in \{0,1\}^n} \mathbf{1}\{x \in \cS(\phi)\} =  \frac{Z(\phi)}{\E_{\text{unif}}[Z(\phi)]}\, ,
 \end{eqnarray*}
by the known closed form of $\E_{\text{unif}}[Z(\phi)] = 2^n (1-2^{-k})^m$, which can be directly derived by linearity. The likelihood-ratio test is therefore indeed $\Psi_\text{\sf LR}(\phi) = \mathbf{1}\{Z(\phi) > \E_{\text{unif}}[Z(\phi)] \}$. It is now sufficient to prove $\fP_{\text{unif}}(\Psi(\phi)=1) + \fP_{\text{planted}}(\Psi(\phi)=0)\le \delta$,
as the maximum of two nonnegative numbers is smaller than their sum. By definition of the likelihood-ratio test,
$$\fP_{\text{unif}}(\Psi_\text{\sf LR}(\phi)=1) + \fP_{\text{planted}}(\Psi_\text{\sf LR}(\phi)=0)=1-d_{TV}(\fP_{\text{unif}},\fP_{\text{planted}})\, .$$
Furthermore, by definition of the total variation distance
\begin{eqnarray*}
d_{TV}(\fP_{\text{unif}},\fP_{\text{planted}}) &=& \sum_{\substack{\phi \in \cF^k_{m,n} \\ \fP_{\text{unif}}(\phi) > \fP_{\text{planted}}(\phi)}}\{\fP_{\text{unif}}-\fP_{\text{planted}}\}(\phi)\\
&=& \sum_{\substack{\phi \in \cF^k_{m,n} \\ Z(\phi)/\E[Z]<1}}\Big(1-\frac{Z(\phi)}{\E[Z]}\Big) \fP_{\text{unif}}(\phi)\\
&=& \E_{\text{unif}}\Big[\Big(1-\frac{Z(\phi)}{\E[Z]}\Big)_+ \Big]\, .
\end{eqnarray*}

The total variation distance between distributions of i.i.d. elements being non-decreasing in the sample size, we obtain by Lemma~\ref{LEM:divfinite} that in the square-root regime, for $C$ large enough and $m \ge C\sqrt{n}$, 
$$d_{TV}(\fP_{\text{unif}},\fP_{\text{planted}}) \ge (1-e^{-\gamma_{k}C^2/C_0})(1-C_0/C^2)\, .$$
This bound yields the desired result for some large enough constant $C_{k,\delta}>0$.
\end{proof}

The proof of this theorem indicates that it is possible to distinguish the two distributions whenever $Z$ is not concentrated around its expectation under the uniform distribution. 
Our result is a consequence of the following lemma, that states that in the square-root regime, for a constant $C$ large enough, the ratio $Z/\E[Z]$ is much smaller than 1, with high probability.
\begin{lemma}
\label{LEM:divfinite}
For all $k \ge 2$, $C_0$ an absolute constant, $m = C \sqrt{n}$, and $C, n$ large enough, it holds with probability $1-C_0/C^2$,for some constant $\gamma_k>0$ that
$$Z <  e^{- \gamma_k C^2/C_0} \, \E[Z]\, .$$
\end{lemma}

A stronger result, concerning the linear regime, can be derived similarly in order to answer a question regarding the behavior of $Z$ with respect to its expectation. It is known \cite{AbbMon10} that for $\Delta$ small enough and $n \rightarrow + \infty$, $n^{-1} \log(Z)$  and $n^{-1} \E[\log(Z)]$ have the same limit, called the {\em quenched} average. In the following lemma, we prove that this limit is actually different from the constant $n^{-1} \log(\E[Z])$, called the {\em annealed} average, for all $\Delta>0$. 

\begin{lemma}
\label{LEM:divexp}
For all $k \ge 2$, $\Delta>0$, and $m = \Delta n$ large enough, if $\phi \sim \fP_{\text{unif}}$, it holds with probability $1-o(1)$, for some constant $c_{k,\Delta}>0$ that
$$Z <  e^{- c_{k,\Delta} n} \, \E[Z]\, .$$
\end{lemma}

This result is tangential to the problem at hand but of interest in and of itself. We show here that the quenched and annealed averages are different for all $\Delta$ and $k$, with a gap greater than $c_{k,\Delta}$, for which we give no explicit formula. This phenomenon is hinted at in \cite{AchCoj08,Coj09}, and proven to hold for $\Delta$ large enough in \cite{CojPan12}, with an explicit lower bound for the gap.  We provide a proof for Lemma~\ref{LEM:divfinite} and~\ref{LEM:divexp} in Appendix~\ref{SEC:app}.


\subsection{Information-theoretic lower bound}

The proof of Theorem~\ref{THM:uppersqr} also hints at a lower bounds for the statistical problem. The total variation distance $d_{TV}$ between the {\em uniform} and {\em planted} distributions is close to 0 (and the statistical problem is impossible) when $Z(\phi)$ is concentrated around its expectation. 

The number of satisfying assignments is actually equal to its expectation 
whenever no variable appears in two different clauses. Indeed, when this is the case, the set of satisfying assignments can be described thus. There are $m$ clauses on $m$ distinct groups of $k$ distinct variables. Each clause allows a specific group of $k$ variables to take $2^k-1$ values, and the $n-km$ remaining variables are free. There are therefore $(2^k-1)^m$ possible values for the constrained variables and $2^{n-km}$ possible values for the $n-km$ remaining. Overall, $Z = (2^k-1)^m 2^{n-km} = 2^n (1-2^{-k})^m = \E[Z]$. This observation yields the following lower bound.

\begin{theorem}
\label{THM:lowbound}
For $\nu \in (0,1/2)$, $m\le 2\sqrt{\nu n}/k$, and $m,n$ large enough, it holds that
$$\inf_{\Psi} \big\{\fP_{\normalfont \text{unif}}(\Psi(\phi)=1) \vee \fP_{\normalfont \text{planted}}(\Psi(\phi)=0) \big\} \ge \frac 12 -\nu\, .$$
\end{theorem}\begin{proof}
We use the total variation bound, for any test $\Psi$
\begin{eqnarray*}
\fP_{\text{unif}}(\Psi(\phi)=1) \vee \fP_{\text{planted}}(\Psi(\phi)=0) &\ge&  \frac{1}{2}\big(\fP_{\text{unif}}(\Psi(\phi)=1) + \fP_{\text{planted}}(\Psi(\phi)=0)\big)\\
 &\ge& \frac{1-d_{TV}(\fP_{\text{unif}},\fP_{\text{planted}})}{2}\, .
\end{eqnarray*}
We denote by $F$ the set of formulas where no variable appears in two different clauses.
\begin{eqnarray*}
d_{TV}(\fP_{\text{unif}},\fP_{\text{planted}}) &=& \frac{1}{2}\sum_{\phi \in \cF^k_{m,n}}| \fP_{\text{unif}}-\fP_{\text{planted}}|(\phi)\\
&=& \frac{1}{2}\sum_{\phi \in F}| \fP_{\text{unif}}-\fP_{\text{planted}}|(\phi) + \frac{1}{2}\sum_{\phi \in F^c}| \fP_{\text{unif}}-\fP_{\text{planted}}|(\phi)\\
&=&\frac{1}{2} \sum_{\phi \in F}\Big|\frac{Z(\phi)}{\E[Z]}-1\Big| \fP_{\text{unif}}(\phi) + \frac{1}{2}\sum_{\phi \in F^c}| \fP_{\text{unif}}-\fP_{\text{planted}}|(\phi)
\end{eqnarray*}
As noticed above, for all $\phi \in F$, $Z(\phi) = \E[Z]$; the likelihood ratio is equal to 1. The first term of this equation is therefore equal to 0. This also implies that  $\fP_{\text{unif}}(\phi)=\fP_{\text{planted}}(\phi)$ for all $\phi \in F$, and $\fP_{\text{unif}}(F)=\fP_{\text{planted}}(F)$. The second term is thus upper bounded by $\fP_{\text{unif}}(F^c)=\fP_{\text{planted}}(F^c)$. 
It is sufficient to prove that $\fP_{\text{unif}}(F^c)\le 2\nu$, a variant of the ``birthday problem'': We place a group of $k$ balls in $n$ distinct bins uniformly at random, $m$ times independently. The probability that none of these $m$ groups intersect is equal to $\fP_{\text{unif}}(F)$. When $i$ groups have already been drawn, occupying $ki$ bins, the probability that one of the next $k$ balls falls in an occupied bin is smaller than $k^2 i/n$ (the expected number of such collisions). As $k^2(m-1)/n<1/2$ (for fixed $\nu$ and $n$ large enough) the following holds
$$\fP_{\text{unif}}(F) \ge \prod_{i=1}^{m-1}\Big(1-\frac{k^2 i}{n} \Big) > \prod_{i=1}^{m-1}e^{-2k^2i/n} = e^{-k^2(m-1)(m-2)/n}>1-k^2m^2/n\, .$$
This gives the desired result. 
\end{proof}

From the last two theorems, we can conclude that the {\em optimal rate of detection} is $m^* = \sqrt{n}$. When $m=C \sqrt{n}$, detection is possible with probability of error smaller than $\delta$, for $C$ greater than some constant $\bar C_{k,\delta}$, by using the likelihood-ratio test. It is impossible to distinguish the two hypotheses with error probability smaller than $1/2 -\nu$ for $C<\underline C_{k,\nu}:=2\sqrt{\nu}/k$. No effort has been made to optimize (or even quantify) the constant $\bar C_{k,\delta}$, as a function of $k$ and $\delta$.

%

\section{Polynomial-time testing}
\label{SEC:poly}

For $k \geq 2$, computing the outcome of the likelihood-ratio test involves solving a $\#${\sf P}-complete problem \cite{Val79}, and for $k\ge 3$, even computing the outcome of the satisfiability test $\Psi_\text{\sf SAT}$ (which is already suboptimal) is equivalent to solving a {\sf NP}-hard problem. The testing methods described in the previous section are not computationally efficient: determining if a formula is satisfiable is the {\em quintessential} hard problem, the first known to be {\sf NP}-complete \cite{Coo71,Lev73}, at the root of the web of problems known to be in the same class \cite{Kar72}. None of the tests described above can be computed in a computationally efficient manner. It is therefore legitimate to examine the performance of tests that can be computed in polynomial time. 

Finding a satisfying assignment in formulas that are known to be satisfiable has been the focus of substantial efforts \cite{BraMezZec02,Fla02,KriVil06,CojKriVil07}. A polynomial-time algorithm that does so in the linear regime (for a large enough $\Delta$) is presented in \cite{CojKriVil07}, for the case $k=3$ (their results extend to any fixed $k$). A similar problem is studied as well in \cite{FelPerVem13}. This method can be used as a tool for detection: in the unsatisfiable regime (when $\Delta$ is large enough), the existence of a satisfying assignment is a sufficient reason to reject the null. The main issue of this approach is that the regime of detection is not optimal: $m$ needs to be of order $n$ (linear regime), when only $\sqrt{n}$ (square-root regime) is required for the likelihood-ratio test.



\subsection{Variable coupling test}

The proof that the likelihood-ratio test has a low probability of error in the optimal regime is based on the fact that there is a large number of variables that appear more than once, and on the fact that under the null distribution, a couple of literals based on the same variable have equal probability to have the same sign or opposite signs. We can use this fact to design a test that runs in polynomial time and achieves the optimal rate of detection.

We recall that in each clause, the literals are given in a uniformly random order. Let $T$ be the number of variables (among the $n$ possible) that appear more than once as the first literal of a clause of $\phi$ (according to the random ordering in the data) and $P$ (resp. $D$) the number of those for which the first two occurrences (according to the natural order of the clauses) of the same variable have the same sign (resp. different signs), so that $P+D=T$. The following holds

\begin{theorem}
\label{THM:coutest}
For all $k \ge 2$,  $m,n>0$ and $\delta \in(0,1)$, denote $\Psi_\text{\sf COU}$ the test defined by
$$\Psi_\text{\sf COU}(\phi) = \mathbf{1}\{P/T > 1/2+1/[2 (2^k-1)]^2 \}\, ,$$
and 
$$\tilde C_{k,\delta}:=[2(2^k-1)] ^2 \sqrt{ 2 \log(2/\delta)} \vee \sqrt{1024/\delta}\, .$$
For $m \ge \tilde C_{k,\delta} \, \sqrt{n}$, it holds
$$\fP_{\normalfont \text{unif}}(\Psi_\text{\sf COU}(\phi)=1) \vee \fP_{\normalfont \text{planted}}(\Psi_\text{\sf COU}(\phi)=0)\le \delta\, .$$
\end{theorem}

\begin{proof}

For each variable that appears at least twice as the first literal of a clause, consider the probability that the two first occurrences (according to the natural order of the clauses) of a variable as the first literal of a clause (according to the random ordering in the data) have the same value. It is equal to 1/2 under the uniform distribution, and conditionally on the value of $T$, $P \sim \cB(T,1/2)$. Under the planted distribution, each literal has independently probability $(1+1/(2^k-1))/2$ to have the same value as the corresponding variable in $x_i^*$, and probability $(1-1/(2^k-1))/2$ to have a different value. Overall, the probability that these two literals have the same sign under the planted distribution is
$$\frac 14 \big(1+\frac{1}{2^k-1}\big)^2 + \frac 14 \big(1-\frac{1}{2^k-1}\big)^2 = \frac{1}{2} + \frac{1}{2 (2^k-1)^2}\, .$$
Therefore, conditionally on the value of $T$, $P$ has distribution $\cB(T,1/2+1/[2(2^k-1)^2])$. By Hoeffding's inequality, the following holds for all $\varepsilon>0$
\begin{eqnarray*}
&&\fPu \big( P/T > 1/2 + \varepsilon  \, | \, T\big) \le \exp(-2 \varepsilon^2 T)\\
&&\fPp \big( P/T < 1/2 + 1/[2(2^k-1)^2] -\varepsilon  \, | \, T\big) \le \exp(-2 \varepsilon^2 T)\, 
\end{eqnarray*}
By Lemma~\ref{LEM:pair}, and by definition of $\tilde C_{k,\delta}$, $T \ge \tilde C_{k,\delta}^2/4$ with probability at least $1-\delta/2$. Let $\varepsilon = 1/[2(2^k-1)]^2$, and condition on the event $T \ge \tilde C_{k,\delta}^2/4$. The previous yields, for $C_{k,\delta} \ge \sqrt{2 \log(2/\delta)} / \varepsilon$
\begin{eqnarray*}
&&\fPu \big( P/T > 1/2 + 1/[2(2^k-1)]^2  \, | \, T\big) \le \delta/2\\
&&\fPp \big( P/T < 1/2 + 1/[2(2^k-1)]^2  \, | \, T\big) \le \delta/2\, .
\end{eqnarray*}
Which gives the desired result by a simple union bound.
\end{proof}
\subsection{Hardness hypothesis on random instances}

The result of Theorem~\ref{THM:coutest} can be contrasted with a hypothesis by Feige, formulated in \cite{Fei02}, to prove hardness of approximation results in the worst case. We recall the proposed assumption on the hardness of determining the satisfiability of $3$-{\sf SAT} formulas on average:\\

{\em ``Even when $\Delta$ is an arbitrarily large constant independent of $n$, there is no polynomial time algorithm that refutes most 3CNF formulas with $n$ variables and $m=\Delta n$ clauses, and never wrongly refutes a satisfiable formula.''}\\

Formally, in a statistical language, it is conjectured in this hypothesis that for all $\Delta>0$, in the linear regime, there is no test $\Psi$ that runs in polynomial time such that $\fP_{\text{unif}}(\Psi=1) \le 1/2$, and $\fP_1(\Psi = 0) = 0$, for any distribution $\fP_1$ supported on {\sf SAT}. In particular, in our testing problem, this hypothesis states that no test that runs in polynomial time has a type I error smaller than 1/2 and a type II error equal to 0. At first sight, this is in apparent contradiction with theorem~\ref{THM:coutest}. 
Interestingly, this result shows that up to the optimal square-root regime it is possible to design a test with small type I and type II errors simultaneously, even though it is conjectured and widely believed that it is impossible to distinguish those distributions with a completely one-sided error. 

There has been a recent interest in the notions of optimal rates for polynomial-time algorithms. More specifically, there is a growing literature on limitations, beyond those imposed by information theory, to the statistical performance of computationally efficient procedures. Such phenomena have been hinted at \cite{DecGolRon98,Ser00,ChaJor13,ShaShaTom12}, and studied in specific computational models, such as in \cite{FelGriRey13,FelPerVem13}. More recently, these barriers have been proven to hold for various supervised tasks such as in \cite{DanLinSha13}, based on a primitive on random 3-{\sf SAT} instances, and unsupervised problems in statistics in \cite{BerRig13} and the subsequent \cite{MaWu13,Che13,WanBerSam14}, based on a hardness hypothesis for the planted clique problem. The above discussion shows the difficulty of using Feige's hypothesis as a primitive to prove computational lower bounds for statistical problems: it does not imply that it is impossible to detect planted distributions in a computationally efficient manner in the linear regime, and is extremely sensitive to the allowed probability of type I and type II errors.



\section{Alternative choices for planting distributions}
\label{SEC:alttest}
The tests described in Theorems~\ref{THM:uppersqr} 
 and \ref{THM:coutest} exploit a fundamental difference between the two considered distributions. Planting a satisfying assignment $x^*\in \{0,1\}^n$ breaks the symmetry of the uniform distribution. The likelihood ratio $Z/\E[Z]$ is affected by the imbalances in interactions between variables. 
 Similarly, the variable coupling test is based on the bias in the signs of chosen literals, under the planted distribution.

This asymmetry is a characteristic of our choice of the planting distribution. In this section, we observe that the rates of detection are different for other natural choices of distribution on {\sf SAT}, the set of satisfiable formulas. Such an example is $\fP_\text{\sf  SAT}$, the uniform distribution on {\sf SAT}. In this new statistical problem, the alternative hypothesis becomes $\tilde H_1:\phi \sim \fP_\text{\sf SAT}$.

It is a fundamentally different statistical problem: its optimal rate of detection is the linear regime $m^*=n$, achieved by the satisfiability test $\Psi_\text{\sf SAT}$. Indeed, as shown in a simple remark in Section~\ref{SEC:pdesc}, this test is successful in the satisfiable part of the linear regime. Furthermore, as $\fP_\text{\sf SAT}$ is the uniform distribution on {\sf SAT}, or $\fPu( \, \cdot \, | \phi \in \text{\sf SAT})$, the total variation distance $d_{TV}(\fP_{\text{unif}},\fP_\text{\sf SAT})$ is equal to $\fPu(\phi \notin \text{\sf SAT})$. As explained before, this probability vanishes to 0 for $\Delta$ small enough, which yields the matching lower bound. From a statistical point of view, this modified hypothesis testing problem is a significantly harder task than the detection of planted satisfiability. 

Among all distributions on satisfiable formulas, the closest in total variation distance to the uniform distribution (and therefore the choice of alternative that yields the hardest statistical problem) is the uniform distribution on {\sf SAT}. Other distributions used to generate formulas that are hard to solve, with hidden solutions (usually, with no immediate asymmetry) as in \cite{AchJiaMoo04,BarHarLeo01,JiaMooStr05,KrzMezZde12} are candidates to create detection problems with optimal rate of detection in the linear regime. Such an example is the uniform distribution on formulas that are {\em not-all-equal}, or {\sf NAE} satisfiable.

\appendix

\section{Proofs of technical results}
\label{SEC:app}
Lemma~\ref{LEM:divfinite} and \ref{LEM:divexp} are a consequence of the following result on the number of variables that appear at least twice in the formula. For simplicity of the proof, we only consider the first literal of each clause, which is sufficient to our objective.
\addcontentsline{toc}{subsection}{Lemma~\ref{LEM:pair}}
\begin{lemma}
\label{LEM:pair}
Let $\phi$ be a random formula of $\cF^k_{m,n}$ with distribution $\fPu$. Let $T$ be the number of variables (among the possible $n$) that appear more than once as the first literal of a clause of $\phi$.
\begin{itemize}
\item Let $\Delta>0$, and $m=\Delta n$. There exists positive constants $\varepsilon_{\Delta}$ and $r_\Delta$ such that
$$\fP(T< \varepsilon_\Delta n) \le \frac{r_\Delta}{n} \, .$$
\item Let $C>0$, and $m=C \sqrt{n}$. It holds that
$$\fP(T< C^2/4) \le \frac{576}{C^2} \, .$$
\end{itemize}
\end{lemma}

\begin{proof}
We prove this deviation bounds in the two regimes.\\

{\bf Linear regime \\}

We first place ourselves in the linear regime $m=\Delta n$. The first literals of the clauses of the random formula can be interpreted as being drawn by independently placing $m$ balls uniformly in $n$ bins, and $T_i$ is the indicator of the event ``there are at least two balls in bin $i$''. This is the complement of having either one or no ball in bin $i$, which yields
$$\E[T_i] = 1- \Big[ \Big(1-\frac 1n\Big)^m + m \Big(1-\frac 1n\Big)^{m-1} \frac 1n\Big] = 1- \Big[ \Big(1-\frac \Delta m\Big)^m + \Delta \Big(1-\frac \Delta m \Big)^{m-1} \Big] \, ,$$
which has limit $1-(1+\Delta)e^{-\Delta} = 2 \varepsilon_\Delta>0$. Therefore, for $m$ large enough, $\E[T_i] >\varepsilon_\Delta$. By, definition $T$ and $T_i$, we have
$$T = T_1+\ldots+T_n\, .$$
Therefore, it holds $\E[T] = \E[T_1+\ldots+T_n] > n \varepsilon_\Delta$. These variables are not independent and the variance is less simple
$$\var[T] = n \var[T_1]+ n(n-1) \big[\E[T_1 T_2] - \E[T_1] \E[T_2] \big]\, .$$
We control the last term
\begin{eqnarray*}
\E[T_1 T_2] &=& \fP[ T_1 =1, T_2=1] = \fP[T_1 = 1 | T_2 =1] \fP[T_2=1]\\
&=& \fP[T_1 = 1 | T_2 =1] \E[T_2]\\
&=& \Big[1- \Big[ \Big(1-\frac 1n\Big)^{m-2} + (m-2) \Big(1-\frac 1n\Big)^{m-3} \frac 1n\Big] \Big] \E[T_2]
\end{eqnarray*}
Therefore, we obtain the bound
$$\E[T_1 T_2] - \E[T_1] \E[T_2] \le \Big[ 1- \Big(1-\frac 1n\Big)^2 + \Delta \Big(1-\Big(1-\frac 1n\Big)^2 \Big)\Big] \E[T_2] \le \frac{3+3\Delta}{n}\, .$$
Overall, this yields $\var[T] \le (4+3\Delta) n$. We now apply Chebyshev's inequality, with $r_\Delta = (3+3\Delta)/(\E[T_1]-\varepsilon_\Delta)^2$
$$\fP[ T < \varepsilon_\Delta n] \le \frac{\var[T]}{(\E[T_1]-\varepsilon_\Delta)^2 n^2} \le \frac{r_\Delta}{n}\, .$$

{\bf Square-root regime\\}

This proof is a simple modification of the proof of the linear regime with the same notations, for $m = C\sqrt{n}$. We derive the expectation and variance of $T$
\begin{eqnarray*}
\E[T_i] &=& 1- \Big[ \Big(1-\frac 1n\Big)^m + m \Big(1-\frac 1n\Big)^{m-1} \frac 1n\Big] \\
&=& 1- \Big[ \Big(1-\frac 1 n\Big)^{C \sqrt{n}} + \frac{C}{\sqrt{n}} \Big(1-\frac 1 n \Big)^{C\sqrt{n}-1} \Big] \\
&=&1-\Big[ 1-\frac{C}{\sqrt{n}} + \frac{C^2}{2n} + o\Big(\frac 1n \Big) + \frac{C}{\sqrt{n}} - \frac{C^2}{n} + o\Big(\frac 1n \Big) \Big] = \frac{C^2}{2n}+ o\Big(\frac 1n \Big)\, .
\end{eqnarray*}

Therefore, for $n$ large enough $\E[T_i] \in (C^2/3n,C^2/n)$ and $\E[T_i] \in (C^2/3,C^2$). For the variance, as in the linear regime it holds
$$\var[T] = n \var[T_1]+ n(n-1) \big[\E[T_1 T_2] - \E[T_1] \E[T_2] \big]\, .$$
We obtain in a similar way the following bound, for $n$ large enough
$$\E[T_1 T_2] - \E[T_1] \E[T_2] \le \Big[ 1- \Big(1-\frac 1n\Big)^2 + \frac{C}{\sqrt{n}} \Big(1-\Big(1-\frac 1n\Big)^2 \Big)\Big] \E[T_2] \le \frac{3}{n}\times C^2/n\, .$$

Therefore, $\var[T] \le 4C^2$, and we have, using Chebyshev's inequality
$$\fP[ T \ge C^2/4] \le \frac{\var[T]}{(C^2/3-C^2/4)^2} \le \frac{576}{C^2}\, .$$
\end{proof}
\begin{proof}[Proof of Lemma~\ref{LEM:divfinite} and \ref{LEM:divexp}]
\addcontentsline{toc}{subsection}{Proof of Lemma~\ref{LEM:divfinite} and \ref{LEM:divexp}}

For all $x \in \{0,1\}^n$, $x \in \cS(\phi)$ if and only if $x$ satisfies all the clauses of $\phi$. We can therefore write
\begin{eqnarray*}
Z = \sum_{x \in \{0,1\}^n} \prod_{i=1}^m \mathbf{1}\{x \in \cS(C_i)\}\, .
\end{eqnarray*}
We recall that this yields, for $\phi$ drawn uniformly $\E[Z] =  2^n  (1-2^{-k})^m$.\\

In the proof of Theorem~\ref{THM:lowbound}, we use that $Z$ is equal to its expectation when the $km$ variables in the formula are distinct. In the linear regime, or in the square-root regime for a large enough constant, it is not the case, with high probability. The interactions between the clauses that share the same variable will create an imbalance between couples of clauses where the same variables appears with the same sign, and those where it appears with a different one.\\

 We compute the conditional expectation of $Z$, given the first variable of each clause, and whether the first two occurrences of every variable (when there are two or more) are the same literal or not. Formally, we denote $G = (G_1,\ldots,G_n)$ the partition of $\{1,\ldots,m\}$ in $n$ sets (allowing some of them to be empty), where 
$$G_i = \big\{ j \in \{1,\ldots,m\} : C_j(x) \in \{ x_i \wedge \ldots , \bar x_i \wedge \ldots\} \big\}\, ,$$
and $\sigma = (\sigma_1, \ldots, \sigma _n)$, where $\sigma_i = 0$ if there are less than two elements in $G_i$, $\sigma_i = 1$ if the first two elements of $G_i$ have the same first literal (either both $x_i$ or both $\bar x_i$), and $\sigma_i=-1$ otherwise. By linearity of expectation, it holds
$$\E[Z \, | \, (G, \sigma)] = \sum_{x \in \{0,1\}^n} \E\Big[ \mathbf{1}\{x \in \cS(\phi)\}   \, | \, (G, \sigma)\Big]\, .$$
We now observe that this conditional expectation is constant, for all $x \in \{0,1\}^n$. Indeed, let $e_0$ be the assignment of all zeroes, and $t_x$ be the literal-flipping transformation such that $t_x(e_0)=x$, and $T_x$ the corresponding  literal-flipping transformation on formulas. For all $x$, it holds
$$\phi(x) = \phi(t_x(e_0)) = (T_x \phi) (e_0)\, .$$
For all $x$, $T_x \phi$ also has distribution $\fP_{\text{unif}}$, and $(G,\sigma)$ is invariant by this transformation. Therefore, it holds
\begin{eqnarray*}
\E[Z \, | \, (G, \sigma)] &=& \sum_{x \in \{0,1\}^n} \E\Big[ \mathbf{1}\{x \in \cS(\phi)\}   \, | \, (G, \sigma)\Big]\\
&=& \sum_{x \in \{0,1\}^n} \E\Big[ \mathbf{1}\{e_0 \in \cS(T_x\phi)\}   \, | \, (G, \sigma)\Big]\\
&=& 2^n \E\Big[ \mathbf{1}\{e_0 \in \cS(\phi)\}   \, | \, (G, \sigma)\Big]\, .
\end{eqnarray*}

The assignment $e_0$ will satisfy the formula $\phi$ if and only if it satisfies all the sub-formulas $\phi_{G_1},\ldots,\phi_{G_n}$ (the empty formula is always satisfied). Given $(G,\sigma)$, the events $\{e_0 \in \cS(\phi_{G_i})\}$ are independent: the sub-formulas are satisfied by $e_0$ if and only if every clause contains at least one negated literal, which occurs independently, conditioned on $(G,\sigma)$. We can therefore compute the conditional expectation
\begin{eqnarray*}
\E\Big[ \mathbf{1}\{e_0 \in \cS(\phi)\}   \, | \, (G, \sigma)\Big] &=& \E\Big[ \prod_{i=1}^n \mathbf{1}\{e_0 \in \cS(\phi_{G_i})\}   \, | \, (G, \sigma)\Big]\\
&=& \prod_{i=1}^n \E\Big[\mathbf{1}\{e_0 \in \cS(\phi_{G_i})\}   \, | \, (G, \sigma)\Big]\\
&=& \prod_{i=1}^n \E\Big[\mathbf{1}\{e_0 \in \cS(\phi_{G_i})\}   \, | \, (G_i, \sigma_i)\Big]\\
\end{eqnarray*}
The product terms can be expressed as a function of $g_i =|G_i|$. If $\sigma_i=0$, in the case of $g_i<2$, treating separately the cases $g_i = 0$ or $1$, we have
$$\E\Big[\mathbf{1}\{e_0 \in \cS(\phi_{G_i})\}   \, | \, (G_i, \sigma_i=0)\Big] = \Big(1- \frac{1}{2^k} \Big)^{g_i}\, .$$
If there are at least two elements in $G_i$, we have
\begin{eqnarray*}
&&\E\Big[\mathbf{1}\{e_0 \in \cS(\phi_{G_i})\}   \, | \, (G_i, \sigma_i=1)\Big] = \frac{1}{2} \Big[1 + \Big(1-\frac{1}{2^{k-1}} \Big)^2 \Big]\Big(1- \frac{1}{2^k} \Big)^{g_i-2}\\
&&\E\Big[\mathbf{1}\{e_0 \in \cS(\phi_{G_i})\}   \, | \, (G_i, \sigma_i=-1)\Big] = \Big(1-\frac{1}{2^{k-1}} \Big) \Big(1- \frac{1}{2^k} \Big)^{g_i-2}\, .
\end{eqnarray*}
Overall, this yields
$$\E\Big[\mathbf{1}\{e_0 \in \cS(\phi_{G_i})\}   \, | \, (G_i, \sigma_i)\Big] = \Big[1 +\frac{\sigma_i}{2^{2k}(1-2^{-k})^2}\Big] \Big(1- \frac{1}{2^k} \Big)^{g_i}\, .$$
Recall that we denote $P$ (resp. $D$) the number of groups for which $\sigma_i=1$ (resp. $-1$). It holds that
$$\E[Z \, | \, (G, \sigma)] = 2^n \Big(1- \frac{1}{2^k} \Big)^m \Big[1 +\frac{1}{2^{2k}(1-2^{-k})^2}\Big]^{P}\Big[1 -\frac{1}{2^{2k}(1-2^{-k})^2}\Big]^{D}\, .$$

It is possible to design a set of $(G,\sigma)$, event of probability close to 1, for which this expectation has the desired value. To do so, we study the behavior of $P$ and $D$, the number of variables that appear at least twice among the first variables of the clauses, for which respectively $\sigma_i=1$ or $-1$.  \\ 

Indeed, for a large $T=P+D$, with $P$ and $D$ close to $(P+D)/2$, this expectation is significantly smaller than $\E[Z]$. Indeed, for all $t \in (0,1)$, the function $f_t: \alpha \mapsto (1+t)^{1+\alpha}(1-t)^{1-\alpha}$ is continuous and $f_t(0)=1-t^2$, so there exists $\alpha_t \in (0,1)$ such that $f_t(\alpha)<1-t^2/2$ for all $|\alpha|<\alpha_t$. Therefore, there exists $\alpha_k \in (0,1)$ such that 
$$\Big[1 +\frac{1}{2^{2k}(1-2^{-k})^2}\Big]^{1+\alpha}\Big[1 -\frac{1}{2^{2k}(1-2^{-k})^2}\Big]^{1-\alpha}< 1- \frac{1}{2^{4k+1}(1-2^{-k})^4}:= e^{-\gamma_k}\, ,$$
for all $|\alpha|<\alpha_k$, for some $\gamma_k>0$.\\

For every variable, we denote $T_i = |\sigma_i| \in \{0,1\}$, and $T= T_1+\ldots+T_n$. We now prove independently the two lemmas.\\

{\bf Linear regime, Lemma~\ref{LEM:divexp}\\}

We control $P$ and $D$ in the regime $m=\Delta n$. By lemma~\ref{LEM:pair}, it holds that
$$\fP[ T < \varepsilon_\Delta n] \le  \frac{r_\Delta}{n}\, .$$

Of these $T$ variables, between $T/2(1 + \alpha_k)$ and $T/2(1-\alpha_k)$ will have their first two occurrences with the same literal, with probability greater than $1-e^{-\alpha_k^2 \varepsilon_\Delta n/2}$, by Hoeffding's inequality. We call $B$ the event $T \ge n \varepsilon_\Delta$ and $P \in (T/2(1 - \alpha_k),T/2(1 + \alpha_k) )$. By the above, $\fP( B) = 1-o(1)$. For $(G,\sigma)$ in the event $B$, it holds 
\begin{eqnarray*}
\E[Z \, | \, (G, \sigma)] &=& 2^n \Big(1- \frac{1}{2^k} \Big)^m \Big[1 +\frac{1}{2^{2k}(1-2^{-k})^2}\Big]^{P}\Big[1 -\frac{1}{2^{2k}(1-2^{-k})^2}\Big]^{D}\\
&<& 2^n \Big(1- \frac{1}{2^k} \Big)^m (e^{-\gamma_k})^{T/2} < e^{-\gamma_k \varepsilon_\Delta n /2}\E[Z] :=e^{- 2 c_{k,\Delta} n} \, \E[Z]\, . 
\end{eqnarray*}
Therefore $\E[Z \, | \, B]<  e^{- 2c_{k,\Delta} n} \, \E[Z]$. We can now conclude by conditioning on $B$ and using Markov's inequality
\begin{eqnarray*}
\fP(Z > e^{- c_{k,\Delta} n} \, \E[Z]) &=& \fP(Z > e^{- c_{k,\Delta} n} \, \E[Z] \, | \, B) \fP(B) + \\
&&\fP(Z > e^{- c_{k,\Delta} n} \, \E[Z] \, | \, B^c) \fP(B^c)\\
&\le& \fP(Z > e^{- c_{k,\Delta} n} \, \E[Z] \, | \, B) + \fP(B^c)\\
&\le& \frac{\E[Z \, | \, B]}{e^{- c_{k,\Delta} n} \, \E[Z] } +\fP(B^c)\\
&\le& e^{- c_{k,\Delta} n} + \fP(B^c)\, . 
\end{eqnarray*}
Which yields the desired result.\\

{\bf Square-root regime, Lemma~\ref{LEM:divfinite}\\}

As in the linear regime, we control $P$ and $D$ when $m=C \sqrt{n}$. Lemma~\ref{LEM:pair} yields

$$\fP[ T \ge C^2/4] \le \frac{576}{C^2}\, .$$

Again, of these $T$ variables, between $T/2(1 + \alpha_k)$ and $T/2(1-\alpha_k)$ will have their first two occurrences with the same literal, with probability greater than $1-e^{-\alpha_k^2 C^2/8}$, by Hoeffding's inequality. We call $B$ the event $T \ge C^2/4$ and $P \in (T/2(1 - \alpha_k),T/2(1 + \alpha_k) )$. By the above, $\fP(B) = 1-O(1/C^2)$. For $(G,\sigma)$ in the event $B$, it holds 
\begin{eqnarray*}
\E[Z \, | \, (G, \sigma)] &=& 2^n \Big(1- \frac{1}{2^k} \Big)^m \Big[1 +\frac{1}{2^{2k}(1-2^{-k})^2}\Big]^{P}\Big[1 -\frac{1}{2^{2k}(1-2^{-k})^2}\Big]^{D}\\
&<& 2^n \Big(1- \frac{1}{2^k} \Big)^m (e^{-\gamma_k})^{T/2} < e^{-\gamma_k C^2/8}\E[Z] \, . 
\end{eqnarray*}

Therefore $\E[Z \, | \, B]<  e^{- \gamma_k C^2/8} \, \E[Z]$. We can now conclude by conditioning on $B$ and using Markov's inequality
\begin{eqnarray*}
\fP(Z > e^{- \gamma_k C^2/16} \, \E[Z]) &=& \fP(Z > e^{- \gamma_k C^2/16} \, \E[Z] \, | \, B) \fP(B) +\\
&& \fP(Z > e^{- c_{k,\Delta} n} \, \E[Z] \, | \, B^c) \fP(B^c)\\
&\le& \fP(Z > e^{- \gamma_k C^2/16} \, \E[Z] \, | \, B) + \fP(B^c)\\
&\le& \frac{\E[Z \, | \, B]}{e^{- \gamma_k C^2/16} \, \E[Z] } +\fP(B^c)\\
&\le& e^{- \gamma_k C^2/8} + \fP(B^c)\, . 
\end{eqnarray*}
This yields the second result, for $C$ large enough, and some absolute constant $C_0$.

\end{proof}

\bibliographystyle{amsalpha}
\bibliography{statebib2}

\end{document}